\documentclass[reqno, english]{amsart}

\usepackage{amsfonts}
\usepackage{amsthm}
\usepackage{thmtools}
\usepackage{amsmath}
\usepackage{amssymb}
\usepackage{amscd}
\usepackage[latin2]{inputenc}
\usepackage{t1enc}
\usepackage[mathscr]{eucal}
\usepackage{indentfirst}
\usepackage{graphicx}
\usepackage{graphics}
\usepackage{pict2e}
\usepackage{epic}
\usepackage{bbm}
\usepackage[left=1in,right=1in,top=1in,bottom=1in]{geometry}
\usepackage{epstopdf} 
\usepackage{enumitem}

\usepackage{hyperref}
\hypersetup{
	colorlinks = true,
    linkcolor = blue,
    citecolor = blue,
    urlcolor = black,
}
\usepackage{cleveref}

\theoremstyle{plain}
\newtheorem{thm}{Theorem}
\newtheorem{lem}[thm]{Lemma}

\newtheorem{claim}[thm]{Claim}

\theoremstyle{definition}

\newtheorem*{remark}{Remark}

\newcommand{\N}{\mathbb N}

\newcommand{\cA}{\mathcal{A} }

\newcommand{\pr}{\mathbb{P}}

\newcommand{\E}[0]{\mathbb{E}}

\newcommand{\beq}[1]{\begin{equation}\label{#1}}
\newcommand{\enq}[0]{\end{equation}}

\newcommand{\mn}[0]{\medskip\noindent}
\newcommand{\nin}[0]{\noindent}

\newcommand{\sub}[0]{\subseteq}
\newcommand{\sm}[0]{\setminus}

\newcommand{\0}[0]{\emptyset}

\newcommand{\go}[0]{\omega}
\newcommand{\gO}[0]{\Omega}

\newcommand{\gT}[0]{\Theta}

\newcommand{\eps}[0]{\varepsilon }

\newcommand{\prh}[1][]{\pr_h}

\newenvironment{subproof}[1][\proofname]{
  
  \begin{proof}[#1]
}{
  \end{proof}
}

\DeclareMathOperator{\aut}{aut}
\DeclareMathOperator{\ext}{Ext}

\author{Quentin Dubroff}
\address{Department of Mathematics, Carnegie Mellon University}
\email{qdubroff@andrew.cmu.edu}
\title{Thresholds vs.\ expectation thresholds for non-spanning graphs}

\begin{document} 

\begin{abstract}
    The threshold $p_c(H)$ for the event that the binomial random graph $G_{n,p}$ contains a copy of a graph $H$ is the unique $p$ for which $\mathbb{P}(H \subseteq G_{n,p}) = 1/2$, and the fractional expectation threshold $q_f(H)$ is roughly the best lower bound on $p_c(H)$ using simple expectation considerations. All previously known $H$'s with $p_c(H)$ substantially larger than $q_f(H)$ have the property that $v_H > n/2$ (where $v_H$ is the number of vertices of $H$). We construct small graphs whose threshold for containment in $G_{n,p}$ is of different order than their corresponding fractional expectation threshold: there is a constant $c > 0$ such that for any $m \; (\leq n)$, there is a graph $H$ with $v_H = m$ and $p_c(H) > q_f(H) c \log^{1/2}(v_H).$
\end{abstract}

\maketitle

\section{Introduction}

Let $G_{n,p}$ be the Erd\H{o}s--R\'enyi random graph, the random subgraph of the complete graph on $\{1,2,\dots,n\}$ obtained by keeping each edge with probability $p$. For an increasing event $\cA$ (one not destroyed by the addition of edges), the \emph{threshold} $p_c(\cA)$ is the unique $p$ for which $\pr(G_{n,p} \in \cA) = 1/2$. Thresholds have been a central topic in probabilistic combinatorics since work of Erd\H{o}s and R\'enyi in the early 1960s. Recently, the study of thresholds has been transformed by proofs of the Kahn--Kalai conjecture and its fractional version \cite{FKNP, PP}, powerful statements that furnish estimates on thresholds that are strikingly precise and often easy to compute. In particular, the machinery of \cite{FKNP} associates to any increasing event $\cA$ a relatively simple, first-moment based lower bound $q_f(\cA)$ on $p_c(\cA)$ (called the \emph{fractional expectation threshold}) which is never off by more than a log factor, i.e.
\beq{KK} q_f(\cA) \leq p_c(\cA) < K q_f(\cA) \log \ell(\cA),\enq
where $K$ is an absolute constant and $\ell(\cA)$ is the number of edges in the largest minimal element of $\cA$.

The estimates on $p_c(\cA)$ in \eqref{KK} are very strong (see \cite{FKNP,JPsurvey} for more discussion), but there remain many settings where more precision is needed, for example \cite[Conjecture 1.1]{JKV} or \cite[Conjectures 7.1--7.3]{Talconj} (or more speculatively, the Erd\H{o}s--Rado sunflower conjecture; see \cite{BCW}). 
The purpose of this note is to explore one possibility for refining \eqref{KK} in the context of graph containment in $G_{n,p}$. In the graph containment problem, there is a graph $H$ under consideration, $\cA$ is the event that that $G_{n,p}$ contains a copy of $H$ (see usage notes below), and we simplify notation by writing $p_c(H)$ and $q_f(H)$ instead of $p_c(\cA)$ and $q_f(\cA)$. Here, the parameter $q_f(H)$ from \eqref{KK} takes a simple form:
 \[q_f(H) = \max\{p : \exists\, F \sub H, \; \E_p[X_F] \leq N(H,F)/2\},\]
 where $X_F$ and $N(H,F)$ are respectively the number of copies of $F$ in $G_{n,p}$ and $H$ (see \cite{KKQ} for pointers to the literature and a quick proof that $q_f(H)$ matches the usual, more general definition of $q_f(\cA)$). Note \cite{KKQ, MNSZ} that the inequality $q_f(H) \leq p_c(H)$ is immediate from Markov's inequality since for any $F\sub H$
 \[\pr(H \sub G_{n,p}) \leq \pr(X_F \geq N(H,F)) \leq \E_pX_F/N(H,F).\]
 
Even in the restricted setting of graph containment, both the upper bound and lower bound on $p_c$ in \eqref{KK} can be essentially tight, but it seems that the lower bound is more often the truth. In fact, all previously known examples of $H$'s with $p_c(H)$ much larger than $q_f(H)$ have the property that $H$ is nearly spanning, meaning $v_H > n/2$. This motivates the natural problem, raised independently by (at least) Jeff Kahn, Nina Kam\v{c}ev, and the author, of determining whether nearly spanning graphs are the only graphs for which $p_c(H)/q_f(H)$ can be arbitrarily large. Here we settle this question.
\begin{thm}\label{MT}
    There is $c> 0$ such that for any $m \;(\leq n)$, there is a graph $H$ with $v_H = m$ and
    \beq{lower}p_c(H) > q_f(H) c\log^{1/2}v_H.\enq
\end{thm}
The idea behind \Cref{MT} comes from the informal intuition that the lower bound in \eqref{KK} is essentially correct unless there is some ``coupon collector'' prerequisite for membership in $\cA$. For example in the graph containment problem when $H$ is a perfect matching, where $q_f(H) \asymp 1/n$ and $p_c(H) \asymp \log n/n$ (see \cite{JPsurvey} for a nice exposition of this), the coupon collector prerequisite is that each vertex must lie in an edge of $G_{n,p}$. A proper formalization of this intuition seems likely to be extremely useful.

For the graphs constructed in the proof of \Cref{MT}, the inequality \eqref{lower} is tight up to the \mbox{constant $c$}. However, in light of \eqref{KK}, it is possible that there are non-nearly spanning graphs $H$ with $p_c(H) > q_f(H) c\log v_H$ (and $v_H \rightarrow\infty$), i.e.\ non-nearly spanning graphs achieving the maximum possible gap between the fractional expectation threshold and the actual threshold. It would be nice to see such a graph, but more interesting would be to show that such a graph cannot exist.

\subsubsection*{Usage} For a graph $H$, we use $V(H)$ and $E(H)$ for the sets of vertices and edges of $H$ and $v_H$ and $e_H$ for the respective sizes of these sets. For a subgraph $R \sub H$ and $v\in V(H)$, we write $d_R(v)$ for the degree of $v$ in $R$, the number of edges of $R$ incident to $v$. If $W \sub V(H)$, we write $H[W]$ for the subgraph of $H$ induced by $W$.

As usual, $G_{n,p}$ is the Erd\H{o}s--R\'enyi random graph on vertex set $\{1,\dots, n\}$ obtained by retaining each edge of the $n$-vertex complete graph $K_n$ independently with probability $p$. A copy of a graph $H$ in another graph $G$ is an unlabeled subgraph of $G$ isomorphic to $H$. We use $N(G,H)$ for the number of copies of $H$ in $G$. We (abusively) write $H \sub G_{n,p}$ to refer to the event that $G_{n,p}$ contains a copy of $H$. A labeled copy of $H$ is an injection from $V(H)$ to $V(G)$ that maps edges of $H$ to edges of $G$. 

Asymptotic notation ($O(\cdot), \gO(\cdot), \sim, \go(\cdot), o(\cdot), \gT(\cdot)$) is standard. We sometimes write $f \asymp g$ to mean $f = \gT(g)$.

\subsection*{Acknowledgments} Thanks to Jeff Kahn for many enlightening discussions.

\section{Construction}

Let $F$ be a random $4$-regular graph on $t$ vertices. For every $\{x,y\} \in \binom{V(F)}{2}$, add a new vertex $z$ and the edges $xz$ and $yz$. In the resulting graph $H$, we have $v_H = t + \binom{t}{2}$ and $e_H = 2v_H$.

\begin{lem}\label{ethresh}
There is a constant $L$ such that for any $t$ (with $t + \binom{t}{2} \leq n$), w.h.p.\ (over choices of $F$)
    \[q_{f}(H) \leq L/\sqrt{n}.\]
\end{lem}
\begin{remark}
    To simplify presentation, we omit the simple adjustments needed for construction of $H$ with $v_H$ not of the form $t + \binom{t}{2}$. \Cref{ethresh} is proved in Section~\ref{sec:exp}. The proof is more involved than one would wish, but to give some basic intuition, we note that w.h.p.\ $H$ is balanced, i.e.\ $e_H = 2v_H$ and $e_R \leq 2v_R$ for every $R \sub H$, which is a prerequisite for \Cref{ethresh}. If $t$ is sufficiently large ($t > \log^{1 + \eps} n$ for fixed $\eps > 0$ suffices), then the construction works with $F$ as the empty graph, and \Cref{ethresh} becomes easier, though still not immediate. Similarly, if one is content with proving \Cref{ethresh} with $q_f(H)$ replaced by (the smaller)
\[p_{\E}(H):= \max\{p : \exists\, F \sub H, \; \E_p[X_F] \leq 1/2\},\]
then the proof also somewhat simplifies. 
A conjecture of \cite{KKQ} states that $q_f(H) \asymp p_{\E}(H)$ for every $H$, but the best result to date \cite[Theorem 1.5]{KKQ} says that
\[(p_{\E}(H)\leq\;) \;q_f(H) < O(\log^2 n) p_{\E}(H).\]
\end{remark}

\subsection*{A lower bound on the threshold}

For the remainder of the section, $\pr$ denotes probability in $G = G_{n,p}$ with $p = C/\sqrt{n}$, and $F$ is an arbitrary $4$-regular graph. We will show $\pr(H \subseteq G) = o(1)$ if $C < c\sqrt{\log t}$ for some $c>0$ (as $t \rightarrow \infty$).

\mn 
Say a set $X$  (from $\{1,2,\dots, n\}$) of $t$ vertices is \emph{dangerous} if it contains a copy of $F$. This happens with probability at most $t!p^{e_F}$ by Markov's inequality. 

\mn
Say a set $X$ of $t$ vertices \emph{extends} if every pair of vertices $\{x,y\} \in \binom{X}{2}$ has a distinct common neighbor outside of $X$. So, there is a copy of $H$ in $G$ if and only if there is an $X$ which is dangerous and extends.

\mn
A set $X$ of $t$ distinct vertices is \emph{bad} if
\begin{enumerate}[label=(\alph*)]
    \item\label{supp} $X$ is dangerous and
    \item\label{ext} $X$ extends.
\end{enumerate}

\mn
Let $Z$ count the number of bad sets of $t$ vertices. It is enough to show
\beq{expZ}
\E[Z] < o(1) \quad \text{if} \quad C < c\sqrt{\log t}.
\enq

\begin{proof}[Proof of \eqref{expZ}]
    Since \ref{supp} and \ref{ext} are independent, we have 
    \beq{ez}
	\E[Z] = \binom{n}{t} \pr\ref{supp}\pr\ref{ext}.
	\enq
    
	\nin As noted after the definition of dangerous, Markov's inequality gives
    \[\pr\ref{supp} \leq t!p^{e_F} = t! p^{2t},\]
	
    \nin and we will show
    \beq{extbound}
    \pr\ref{ext} < \exp\left[-\binom{t}{2} e^{-(1 + o(1))C^2}\right].
    \enq
    Using the last two inequalities in \eqref{ez}, we find
    \beq{main}\E[Z] < n^t p^{2t}\pr\ref{ext} < C^{2t}\exp\left[-\binom{t}{2} e^{-(1 + o(1))C^2}\right],\enq
    which is $o(1)$ if $C < c\sqrt{\log t}$ (and $t \rightarrow \infty$) as desired. 
    \end{proof}
    
    To complete the proof we need to establish \eqref{extbound}.
	
	We use an extension of the Van den Berg--Kesten (BK) inequality showing negative correlation for disjoint occurrence of multiple events. Given increasing events $\cA_1,\dots,\cA_k$, their disjoint occurrence is the event
    \[\Box_i\, \cA_i = \{A_1\cup A_2 \cup\dots \cup A_k: A_i \in \cA_i \;\forall i,\; A_i \cap A_j = \0 \text{ if } i \neq j\}.\]
    The following inequality follows immediately from \cite[Theorem 1]{BKBK}, but, as noted in \cite{BKBK}, the special case here can be obtained easily from the BK inequality by induction on $k$. (See also \cite{Reimer} for more on the BK inequality and \cite{thesis} for a short proof.)
    \begin{lem}\label{lem:corr}
        \[\pr(\Box_i \, \cA_i) \leq \prod_i \pr(\cA_i).\]
    \end{lem}
    
    \begin{proof}[Proof of \eqref{extbound}]
	For a $t$-set $X$ and $x,y \in X$, let $\cA_{\{x,y\}}$ be the event that $x$ and $y$ have a common neighbor outside of $X$. Then the event that $X$ extends is the same as
	\[\underset{\{x,y\}\in {\binom{X}{2}}}{\Box} \cA_{\{x,y\}}.\]
    Lemma~\ref{lem:corr} gives
	\beq{T1}
	\pr(\Box_{\{x,y\}} \cA_{\{x,y\}}) \leq \prod_{\{x,y\}} \pr(\cA_{\{x,y\}}). 
	\enq
	Now fix $x,y \in X$ and let $Y$ count the number of common neighbors $\{x,y\}$ outside of $X$. Observe that $Y \sim \text{Bin}(n-t, p^2)$
	so 
	\[\pr(\cA_{\{x,y\}}) = \pr(Y > 0) = 1 - (1 - p^2)^{n-t} < 1 - e^{-(1 + o(1))C^2},\]
	which, plugged into \eqref{T1}, gives \eqref{extbound}.
\end{proof}

\section{The expectation threshold}\label{sec:exp}

    Here we prove \Cref{ethresh}. Throughout this section $p = Ln^{-1/2}$ for some constant $L$ to be determined. Our goal is to show that if $L$ is sufficiently large, then $\E_p[X_R] \geq N(H,R)$ for every $R \sub H$ (w.h.p.\ over choices of $F$). Note that (using $(m)_k = m(m-1)\dots(m-k+1)$ for the falling factorial),
    \beq{orient}\E_p[X_R] = \frac{(n)_{v_R}p^{e_R}}{\aut(R)} > \frac{n^{v_R}p^{e_R}}{e^{v_R}\aut(R)}.\enq
    To simplify matters:
    \begin{enumerate}[label=(\alph*)]
        \item We assume $R$ has no isolated vertices. \emph{Justification:} removing isolated vertices cannot increase $\E_p[X_R]/N(H,R)$ (see \cite[Claim 2.6]{KKQ} for the simple, formal proof).
        \item We work with $\mu_p(R) := n^{v_R}p^{e_R}/\aut(R)$ instead of $\E_p[X_R]$. \emph{Justification:} if $R$ has no isolated vertices and we can show $\mu_p(R) \geq N(H,R)$, then $\E_{e^2p}[X_R] \geq N(H,R)$ because (using \eqref{orient} and the fact that $2e_R \geq v_R$ for $R$'s with no isolated vertices)
        \[\E_{e^2 p}[X_R] > \frac{n^{v_R}(e^2p)^{e_R}}{e^{v_R}\aut(R)} = \frac{n^{v_R}p^{e_R}e^{{2e_R}}}{e^{v_R}\aut(R)} \geq \frac{n^{v_R}p^{e_R}}{\aut(R)} \geq N(H,R).\]
        Thus if \Cref{ethresh} holds with $\mu_p$ replacing $\E_p$ in the definition of $q_f$, then it also holds as currently stated (with $L$ replaced by $e^2L$).
        \item\label{iso} We assume that all components of $R$ are isomorphic. \emph{Justification:} if $R = R_1 \cup R_2$ with $V(R_1) \cap V(R_2) = \emptyset$ and no component of $R_1$ isomorphic to a component of $R_2$, then $\mu_p(R) = \mu_p(R_1)\mu_p(R_2)$ and $N(H,R) \leq N(H,R_1)N(H,R_2)$, so if $\mu_p(R_i) \geq N(H,R_i)$ for each $i$, then 
        \[\mu_p(R) = \mu_p(R_1)\mu_p(R_2)\geq N(H,R_1)N(H,R_2) \geq N(H,R).\]
        (In fact, we could assume that $R$ is connected, as in \cite[Proposition 2.4]{KKQ}, but this stronger assumption doesn't seem to make the proof any easier.)
    \end{enumerate}
    To summarize, our updated goal is to show that if $p = L/\sqrt{n}$ and $L$ is sufficiently large, then
    \beq{goal}
        \frac{n^{v_R}p^{e_R}}{\aut(R)} \geq N(H,R)
    \enq
    for every $R \sub H$ with all connected components of $R$ isomorphic to some graph, say $K$. We also note for future use that the condition $t + \binom{t}{2} \leq n$ implies
    \beq{t2n}
        t < \sqrt{2n}.
    \enq

    We prove \Cref{ethresh} for any $F$ satisfying the conclusion of the following well-known lemma (see e.g.\ \cite[Lemma 9.12]{FriezeKaronski}).
    \begin{lem}\label{lem:sparse}
        There is $\eps > 0$ such that if $F$ is a uniformly random $4$-regular graph on $t$ vertices, then w.h.p.\ 
        \beq{bal}e_{K} \leq  v_{K}\enq
        for every $K \sub F$ with $v_K \leq \eps \log t$. 
    \end{lem}

    \nin \Cref{lem:sparse} will be used to obtain the following statement.
    \begin{lem}\label{lem:inF}
        There is a constant $L$ so that if $p = L/\sqrt{n}$ and $F$ is a $4$-regular graph on $t < \sqrt{2n}$ vertices such that \eqref{bal} holds for every $K \sub F$, then
        \beq{muNaut} n^{v_{F'}}p^{e_{F'}} \geq N(F,F')\aut(F')^2\enq
        for every $F' \sub F$.                                         
    \end{lem}

    We also need some control on the number paths and cycles incident to vertices of $H$. For $k,m \in \N$, let 
    $p_k(m)$ be the maximum, over $v \in V(H)$, number of ways to choose $m$ paths of length $k$ in $H$ starting at $v$, and similarly, let $b_k(m)$ be the maximum, over $(x,y) \in V^2$, number of ways to choose $m$ paths of length $k \geq 2$ in $H$ connecting $x$ and $y$. So if $x = y$, then $b_k(m)$ is twice the number of ways of attaching $m$ $k$-cycles to $x$.
    \begin{claim}\label{cl:count}
        If $p \geq 8n^{-1/2}$, then
        \beq{path}
            p_k(m) m! < (np)^{km},
        \enq
        and
        \beq{bridge}
            b_k(m)m! < (n^{k-1}p^k)^m.
        \enq
    \end{claim}

    \nin We now complete our proof of \Cref{ethresh} by establishing \eqref{goal}, deferring proofs of \Cref{lem:inF} and \Cref{cl:count} to the end of this section. We first prove \eqref{goal} when $R$ has maximum degree two.
    \begin{proof}[Proof of \eqref{goal} when $R$ has maximum degree two]

    We take $L = 8$ and fix $R \sub H$. By \ref{iso} and our assumption on the maximum degree, we may assume that $R$ is the union of $\ell$ copies of a graph $K$, where $K$ is either a $k$-path or $k$-cycle. Let 
    \[d = \begin{cases}
        p_k(1)/2 & \text{if $K$ is a path};\\
        b_k(1)/(2k) & \text{if $K$ is a cycle}.
    \end{cases}\]
    We have
    \[N(H,R) \leq \binom{n}{\ell}d^\ell\]
    since each copy of $R$ may be specified by choosing a vertex incident to each copy of $K$ in $R$ (restricting to endpoints of $K$ when $K$ is a path) and extending each vertex to a copy of $K$, noting that each $k$-path is counted twice and each $k$-cycle is counted $2k$ times (since $b_k$ counts ordered paths). Now, using \Cref{cl:count} and our choice of $L$ for the final inequality,
    \[N(H,R)\aut(R) \leq \binom{n}{\ell}d^\ell \ell! \aut(K)^\ell \leq n^\ell\begin{cases}
        p_k(1)^\ell & \text{if $K$ is a path}\\
        b_k(1)^\ell & \text{if $K$ is a cycle}.
    \end{cases} < (n^{v_K}p^{e_K})^\ell = n^{v_R}p^{e_R},\]
    which proves \eqref{goal}.
    \end{proof}

    \begin{proof}[Proof of \eqref{goal}]
    We fix $R \sub H$, recalling our assumption \ref{iso} that all connected components of $R$ are isomorphic to a graph $K$, and we take $L$ ($\geq 8$) as in \Cref{lem:inF}. To prove \eqref{goal}, 
    we need a suitable upper bound on $N(H,R)\aut(R)$. 
    
    Since we have already proved \eqref{goal} when $K$ is a path or cycle, we may assume that $K$ has a vertex of degree at least three, and we consider
    \[B = \{v \in V(R): d_R(v) > 2\}\]
    and $R_B = R[B]$. We prove an upper bound on $N(H,R)$ by first controlling the number of \emph{labeled} copies of $R_B$ in $H$ and then understanding how each labeled copy of $R_B$ can extend to a copy of $R$ in $H$, where the vertices of this copy outside of the image of $B$ are unlabeled.
    
    In each copy of $R$ in $H$, the image of $B$ must be contained in $V(F)$ (because $d_H(v) = 2$ for $v \not\in V(F)$), so the number of possible images for $B$ is at most the number of labeled copies of $R_B$ in $F$, namely $N(F,R_B)\aut(R_B)$. Let $\ext_H(B;R)$ be the maximum (over embeddings of $R_B$ in $F$) number of ways to extend a labeled copy of $R_B$ to a (partially labeled) copy of $R$ in $H$, and let $\aut(R/B)$ be the number of automorphisms of $R$ that fix $B$ pointwise. We have $\aut(R) \leq \aut(R_B)\aut(R/B)$, and
    \[N(H,R) \leq N(F,R_B)\aut(R_B)\ext(H,R/B).\]
    Thus
    \[\frac{n^{v_R}p^{e_R}}{N(H,R)\aut(R)} \geq \frac{n^{|B|}p^{e_{R_B}}}{N(F,R_B)\aut(R_B)^2}\cdot\frac{n^{v_R - |B|}p^{e_R - e_{R_B}}}{\ext_H(B;R) \aut(R/B)}.\]
    The first fraction on the r.h.s.\ is at least one by \Cref{lem:inF} and our choice of $L$, so we only need to show 
    \[ X:=\frac{n^{v_R - |B|}p^{e_R - e_{R_B}}}{\ext_H(B;R) \aut(R/B)} \geq 1.\]

    Each labeled copy of $R_B$ in $F$ can be extended to a copy of $R$ in $H$ by attaching (pairwise internally disjoint) bare paths (ones with internal degree equal to two) to the vertices of $B$. The contribution to $X$ from attaching $m$ pendant bare paths of length $k$ to some given vertex is at least
    \[\frac{n^{mk}p^{mk}}{p_k(m)m!},\]
    and similarly the contribution to $X$ from attaching $m$ bare paths of length $k$ connecting vertices $v,w$ from $B$ is at least
    \[\frac{(n^{k-1}p^{k})^m}{b_k(m)m!}.\]
    By \Cref{cl:count}, each of these fractions are greater than $1$ when $p \geq 8/\sqrt{n}$, which completes the proof.
    \end{proof}

    \begin{proof}[Proof of \Cref{cl:count}]
    We use $W:= V(H) \sm V(F)$.
    \begin{subproof}[Proof of \eqref{path}]
        It suffices to show that the number of paths of length $k$ starting at a given vertex in $H$ is at most $4^k t^{\lceil k/2 \rceil}$, since then (using \eqref{t2n} for the third inequality)
        \[p_k(m)m! \leq \binom{4^k t^{\lceil k/2 \rceil}}{m}m! \leq 4^{km} t^{m(\lceil k/2 \rceil)} < 4^{km}(\sqrt{2n})^{m(\lceil k/2 \rceil)} < (np)^{km}\]
        if $p \geq  8/\sqrt{n}$.
        To count the number of paths $v_0, v_1,\dots, v_k$ starting at some fixed $v_0 \in V(H)$, we begin specification of a path by choosing the indices $I\sub [k]$ for which the path is in $W$; given $I$, we choose the sequence $(v_i)$ in order by choosing $v_{i+1}$ from one of the $4$ neighbors of $v_i$ in $F$ when $i+1 \not\in I$, and one of the $t-1$ neighbors of $v_i$ in $W$ if $i+1 \in I$. For fixed $|I|$, the total number of choices is at most
        \[(t-1)^{|I|}4^{k - |I|};\]
        Moreover, $W$ is independent, so $|I| \leq \lceil k/2 \rceil$, implying that the total number of choices is at most
        \[\sum_{|I| \leq \lceil k/2 \rceil} (t-1)^{|I|}4^{k - |I|} < 2^k t^{\lceil k/2 \rceil}4^{\lfloor k/2 \rfloor} \leq 4^k t^{\lceil k/2 \rceil},\]
        which completes the proof.
    \end{subproof}
    \begin{subproof}[Proof of \eqref{bridge}]
        It is enough to show that the number of paths of length $k$ going between two vertices in $H$ is at most $4^k t^{\lceil (k-3)/2 \rceil}$, since then (recalling \eqref{t2n} for the third inequality)
        \[b_k(m)m! \leq \binom{4^k t^{\lceil (k-3)/2 \rceil}}{m}m! \leq 4^{km} t^{m\lceil (k-3)/2 \rceil} < 4^{km}(\sqrt{2n})^{m(\lceil (k-3)/2 \rceil)} < (n^{k-1}p^k)^{m},\]
        if $k \geq 2$ and $p \geq 8/\sqrt{n}$.
        To count the number of paths $v_0, v_1,\dots, v_k$ connecting two fixed vertices $v_0, v_k \in V(H)$, we begin by specifying the indices $I\sub [k-3]$ for which $v_{i} \in W$. Given $I$, we choose the first $k-3$ internal vertices of the path in order, noting (as before) that there are at most $4$ choices for $v_{i}$ when $i \not\in I$ and at most $t-1$ choices when $i \in I$. For fixed $|I|$, the total number of choices is thus at most
        \[(t-1)^{|I|}4^{k - 3 - |I|};\]
        Now $|I| \leq \lceil (k-3)/2 \rceil$ because $W$ is independent, so the number of choices for the first $k-3$ internal vertices of the path is at most
        \[\sum_{|I| \leq \lceil (k-3)/2 \rceil} (t-1)^{|I|}4^{k - 3 - |I|} < 2^{k-1} t^{\lceil (k-3)/2 \rceil}4^{\lfloor (k-3)/2 \rfloor} \leq 4^{k-2} t^{\lceil (k-3)/2 \rceil}.\]
        Once the first $k-3$ internal vertices of the path are specified, there are at most $12$ ways to complete the path since the number of length three paths between any two vertices in $H$ is at most $12$. Thus the total number of paths of length $k$ between two vertices is at most $4^{k-2} t^{\lceil (k-3)/2 \rceil}\cdot 12 < 4^k t^{\lceil (k-3)/2 \rceil}$.
    \end{subproof}
    
    \end{proof}

    \begin{proof}[Proof of Lemma~\ref{lem:inF}]
            Fix $F' \sub F$, take $\eps$ as in \Cref{lem:sparse}, and let $L = 64 e^{1/\eps}$. As in \ref{iso}, suffices to prove the lemma under the assumption that all components of $F'$ are isomorphic: if $F' = F_1 \sqcup F_2$ with no component of $F_1$ isomorphic to a component of $F_2$, then $\mu_p(F') = \mu_p(F_1)\mu_p(F_2)$, $\aut(F') = \aut(F_1)\aut(F_2)$, and $N(F,F') \leq N(F,F_1)N(F,F_2)$, so if \eqref{muNaut} holds for $F_1$ and $F_2$, then it also holds for $F'$.

        We need two easy facts (with explanation to follow): If $K$ is a connected subgraph of $F$, then 
        \beq{autK}\aut(K) < 4^{v_K},\enq
        and, recalling $v_F = t$,
        \beq{NSK}
            N(F, K) < t 4^{v_K}.
        \enq
        For justification of both observations, let $J$ be a graph of maximum degree $4$ and fix an ordering $(x_1,\dots, x_k)$ of $V(K)$ such that $K_i = K[\{x_1,\dots, x_i\}]$ is connected for each $i$. The number of labeled embeddings of $K$ in $J$ is at most $v_J \cdot 4 \cdot 3^{v_K - 2}$ since there are at most $v_J$ choices for the image of $x_1$ in $J$, and, having embedded $K_i$ in $F$, there are at most $4$ choices for $x_{i+1}$ if $i = 1$ and at most $3$ choices for $x_{i+1}$ if $i > 1$. When $J = K$, the bound is $v_K \cdot 4\cdot 3^{v_K-2} < 4^{v_K}$, giving \eqref{autK}, and if $J = F$, the bound is $v_F \cdot 4 \cdot 3^{v_K - 2} < t 4^{v_K}$, which implies \eqref{NSK}.

        Now, suppose $F'$ consists of $s$ components isomorphic to a graph $K$. The inequality \eqref{autK} gives
        \[\aut(F') = \aut(K)^s s! < 4^{v_K s} s! = 4^{v_{F'}} s!.\]
        Also,
        \[N(F,F') \leq \binom{N(F,K)}{s} \leq \frac{N(F,K)^s}{s!},\]
        so
        \[N(F,F')\aut(F')^2 < 16^{v_{F'}} N(F,K)^s s!.\]
        Using \eqref{NSK} in the previous bound (as well as $s! \leq s^s$), we get 
        \beq{Naut}
            N(F,F')\aut(F')^2 < 16^{v_{F'}} (t 4^{v_K})^s s^s = 64^{v_{F'}} t^s s^s.
        \enq
        
        To finish, first suppose that $v_K > \eps \log t$, recalling $\eps$ is as in \Cref{lem:sparse}. Since $s = v_{F'}/v_K < v_{F'}/(\eps \log t)$ (and $s \leq t$), we have
        \[(ts)^s \leq t^{2s} < t^{2v_{F'}/(\eps \log t)} = e^{2v_{F'}/\eps},\]
        so in this case \eqref{Naut} implies
        \beq{NFF'1}N(F,F')\aut(F')^2 < 64^{v_{F'}}e^{2v_{F'}/\eps}.\enq
        Noting that $e_{F'} \leq 2v_{F'}$ (since $F' \sub F$ has maximum degree four), we have
        \[n^{v_{F'}}p^{e_{F'}} \geq L^{2v_{F'}},\]
        which is bigger than the r.h.s.\ of \eqref{NFF'1} for our choice of $L$.

        Suppose instead that $1 < v_K \leq \eps \log t$ (the case $v_K = 1$ following directly from \eqref{t2n} for large enough $n$). \Cref{lem:sparse} gives $e_K \leq v_K$, so
        \beq{mu2}n^{v_{F'}}p^{e_{F'}} \geq (L\sqrt{n})^{v_{F'}}.\enq
        But $s \leq v_{F'}/2$ (since $v_K \geq 2$), so (again using $s \leq t$) we have
        \[(ts)^s \leq t^{2s} \leq t^{v_{F'}}.\]
        Using this bound in \eqref{Naut} gives
        \[N(F,F')\aut(F')^2 < 64^{v_{F'}}t^{v_{F'}},\]
        which, noting \eqref{t2n}, is less than the r.h.s.\ of \eqref{mu2} for our choice of $L$.
    \end{proof}

\bibliography{bib.bib}

\begin{thebibliography}{10}

\bibitem{BKBK}
Jacob~D. Baron and Jeff Kahn.
\newblock A natural extension of the {BK} inequality, 2019.
\newblock arXiv:1905.02883.

\bibitem{BCW}
Tolson Bell, Suchakree Chueluecha, and Lutz Warnke.
\newblock Note on sunflowers.
\newblock {\em Discrete Mathematics}, 344(7):112367, 2021.

\bibitem{thesis}
Quentin Dubroff.
\newblock {\em Random walks, the {H}-space of a random graph, and four miniatures}.
\newblock PhD thesis, Rutgers University, 2024.

\bibitem{KKQ}
Quentin Dubroff, Jeff Kahn, and Jinyoung Park.
\newblock On the ``second'' {K}ahn--{K}alai conjecture, 2025.
\newblock arXiv:2508.14269.

\bibitem{FKNP}
Keith Frankston, Jeff Kahn, Bhargav Narayanan, and Jinyoung Park.
\newblock Thresholds versus fractional expectation-thresholds.
\newblock {\em Annals of Mathematics}, 194(2):475--495, 2021.

\bibitem{FriezeKaronski}
Alan Frieze and Micha{\l} Karo{\'n}ski.
\newblock {\em Introduction to random graphs}.
\newblock Cambridge University Press, 2015.

\bibitem{JKV}
Anders Johansson, Jeff Kahn, and Van Vu.
\newblock Factors in random graphs.
\newblock {\em Random Structures \& Algorithms}, 33(1):1--28, 2008.

\bibitem{MNSZ}
Elchanan Mossel, Jonathan Niles-Weed, Nike Sun, and Ilias Zadik.
\newblock On the second {K}ahn--{K}alai conjecture, 2022.
\newblock arXiv:2209.03326.

\bibitem{JPsurvey}
Jinyoung Park.
\newblock Threshold phenomena for random discrete structures.
\newblock {\em Notices of the American Mathematical Society}, 70(10), 2023.

\bibitem{PP}
Jinyoung Park and Huy~Tuan Pham.
\newblock A proof of the {K}ahn-{K}alai conjecture.
\newblock {\em Journal of the American Mathematical Society}, 37:235--243, 2024.

\bibitem{Reimer}
David Reimer.
\newblock Proof of the {V}an den {B}erg--{K}esten conjecture.
\newblock {\em Combinatorics, Probability and Computing}, 9(1):27--32, 2000.

\bibitem{Talconj}
Michel Talagrand.
\newblock Are many small sets explicitly small?
\newblock In {\em Proceedings of the forty-second ACM symposium on Theory of computing}, pages 13--36, 2010.

\end{thebibliography}
\bibliographystyle{plain}

\end{document}